\documentclass{amsart}
\usepackage{amsfonts}
\usepackage{amscd}
\usepackage{euscript}
\newtheorem{theorem}{Theorem}[section]
\newtheorem{lemma}[theorem]{Lemma}
\newtheorem{corollary}[theorem]{Corollary}
\newtheorem{proposition}[theorem]{Proposition}
\theoremstyle{definition}
\newtheorem{question}[theorem]{Question}
\newtheorem{definition}[theorem]{Definition}
\newtheorem{example}[theorem]{Example}

\theoremstyle{remark}
\newtheorem{remark}[theorem]{Remark}

\numberwithin{equation}{section}

\begin{document}

\title{Self-injectivity of $\EuScript{M}(X,\mathcal{A})$ versus $\EuScript{M}(X,\mathcal{A})$ modulo its socle}

\author{A. R. Olfati}
\address{Department of Mathematics, Yasouj University, Yasouj,
Iran}

\email{alireza.olfati@yu.ac.ir, olfati.alireza@gmail.com}






\begin{abstract}
Let $\mathcal{A}$ be a field of subsets of a set $X$ and $\EuScript{M}(X,\mathcal{A})$  be the ring of all real valued $\mathcal{A}$-measurable functions on $X$. It is shown that $\EuScript{M}(X,\mathcal{A})$  is self-injective if and only if $\mathcal{A}$ is a complete and $\mathfrak{c}^+$- additive field of sets. This  answers a question raised in [H. Azadi, M. Henriksen and E. Momtahan,  \textit{Some properties of algebras of real valued measurable functions}, Acta Math. Hungar, 124, (2009), 15--23]. Also, it is observed that if  $\mathcal{A}$ is a $\sigma$-field, $\EuScript{M}(X,\mathcal{A})$ modulo its socle is self-injective if and only if $\mathcal{A}$ is a complete and $\mathfrak{c}^+$- additive field of sets with a finite number of atoms. 
\end{abstract}

\maketitle

\section {Introduction}

\subsection{Ring theoretic concepts}
Throughout this discussion, by  \textit {ring}, we mean a  commutative ring with identity. A ring $R$ is called regular (in the sense of von Neumann) if for every $r\in R$ there exists $t\in R$ with $r=r^2 t$. We say $R$ is a \textit {semi-prime ring} if it has no nilpotent elements other than zero.  If $K$ is a subset of a ring $R$, then we define the \textit {annihilator} of $K$ to be $\text{Ann}(K)=\{r\in R: rK=\{0\}\}$.
If $R$ is a subring of a semi-prime ring $S$, we say that $S$ is a \textit {ring of quotients of} $R$ (in the sense of Lambek \cite{L}) if for each $s\in S\setminus \{0\}$, there is $r\in R$ such that $0\neq rs\in R$. Among the rings of quotients of $R$ there is a maximum,
denoted $Q_{\text{max}}(R)$, which we refer to as the \textit {complete ring of quotients of} $R$.

A commutative ring may be considered as a module over itself in a natural manner. A ring $R$ is called {\textit{\bfseries self-injective}}($\aleph_0$-{\textit{\bfseries self-injective}}) if every module homomorphism $\varphi: I\rightarrow R$ can be extended to a module homomorphism $\Phi: R\rightarrow R$, where $I$ is an ideal (a countably generated ideal). We remark that every self-injective ring is $\aleph_0$-self-injective.   Indeed, $R$ is self-injective if $R$ is an injective object  in the category $R$-\textbf{Mod}. Self-injective rings were introduced by Johnson and Wong in \cite{JW}. In every semi-prime commutative ring $R$ we have the following well-known result, see e.g. \cite[Lemma 2.2]{W}. 

\begin{lemma}\label{3.7}
Let $R$ be a semi-prime commutative  ring, then $R$ is self-injective if and only if $\text{Q}_{\mbox{max}}\left(R\right)=R$.
\end{lemma}

An element $e\in R$ is called an \textit {idempotent} if $e^2=e$. The set of all idempotents of $R$ is denoted by $\text{Id}\left(R\right)$. A ring $R$ is called a {\textit{\bfseries Baer ring }} (resp., {\textit{\bfseries weak Baer ring}})   if for every $K\subseteq R$ (resp., $k\in R$), there is $e\in \text{Id}\left(R\right)$ such that $\text{Ann}(K)=(e)$ (resp., \text{Ann}(k)=(e)). 

The following lemma is proved in \cite[Section 2, Proposition 4]{L}.

\begin{lemma}\label{1.1.1} 
Every self-injective semi-prime ring is regular and Baer.
\end{lemma}

We recall that a subset $T$ of a commutative ring $R$ is said to be \textit {orthogonal}, provided $xy=0$ for all $x,y\in T$ with $x\neq y$. If $S\cap T=\emptyset$ and $S\cup T$ is an orthogonal subset of $R$, an element $a\in R$ is said to \textit {separate} $S$ from $T$ if $s^2 a=s$ for all $s\in S$ and $a\in \text{Ann}(T)$;  see \cite{W}.

In the following lemma, we recall an intrinsic characterization criterion for self-injectivity ($\aleph_0$-self-injectivity); see \cite{K1} and \cite{W} for more details.

\begin{lemma}\label{3.8}
Let $R$ be a semi-prime commutative  ring. The following statements are equivalent.

\begin{enumerate}
\item $R$ is self-injective ($\aleph_{0}$-self-injective).
\item $R$ is regular and whenever $S\cup T$ is an orthogonal (countable orthogonal) subset of $R$ with $R\cap S=\emptyset$, then there is an element in $R$ which separates $S$ from $T$.
\end{enumerate}
\end{lemma}

\subsection{Boolean algebras and field of sets}
In this subsection we assume familiarity with Boolean algebras. For the required background materials see e.g. \cite{GH, KO, Si}.
Let $\Bbb{B}$ be a Boolean algebra and for $a,b\in \Bbb{B}$, by $a\vee b$, $a\wedge b$ and $a^\prime$ denote the Boolean operations of join, meet and complementation, respectively. If $a\vee b=b$ we write $a\leq b$. The relation $\leq$ is a partially order on $\Bbb{B}$. A non zero element $b\in \Bbb{B}$ is called an {\textit{\bfseries atom}} if  for every nonzero $c\in \Bbb{B}$,  the inequality $c\leq b$, implies $c=b$. The set of all atoms of $\Bbb{B}$ is denoted by $\text{At}\left(\Bbb{B}\right)$.  A Boolean algebra $\Bbb{B}$ is called \textit{atomic} if every non zero element of $\Bbb{B}$ contains an atom. Assume that $\kappa$ is an infinite  cardinal number. A Boolean algebra $\left(\Bbb{A}, \vee, \wedge, \prime \right)$ is called {\textit{\bfseries $\kappa^+$-complete}}  if  every non empty subset $\Bbb{K}\subseteq \Bbb{A}$ with $|\Bbb{K}|\leq \kappa$ has a supremum, i.e., there is $x\in \Bbb{A}$ such that:
\begin{itemize}
\item $k\leq x$ for every $k\in \Bbb{K}$,
\item if for $y\in \Bbb{B}$, $k\leq y$ for any $k\in \Bbb{K}$, then $x\leq y$.
\end{itemize}
A Boolean algebra is called \textit{complete} if it is $\kappa^+$-complete for every cardinal number $\kappa$. $\aleph_0^+$-complete Boolean algebras are called also $\sigma$-complete algebras.

Let $\Bbb{B}$ be a Boolean algebra. A non empty subset $I\subseteq \Bbb{B}$ is called an order ideal of $\Bbb{B}$ if 
\begin{itemize}
\item  for each $x\in I$ and each $y\in \Bbb{B}$, the inequality $y\leq x$ implies  $y\in I$.
\item  if $x, y\in I$, then $x\vee y\in I$.
\end{itemize}
For a Boolean algebra $\left(\Bbb{B}, \vee, \wedge, \prime, 0, 1 \right)$, an order ideal $I$ is called \textit{proper} if $1\notin I$.
A proper order ideal $P$ is called prime if for every $b\in \Bbb{B}$ we either have $b\in P$ or $b^\prime \in P$.

Recall that for a Boolean algebra $\Bbb{B}$,  we denote by $\text{Spec}\left(\Bbb{B}\right)$,  the collection of  all prime ideals of $\Bbb{B}$. The set  $\text{Spec}\left(\Bbb{B}\right)$ can be equipped with
the Stone topology. The collection of all sets of the form
\[\mathcal{U}(a)=\{I \in \text{Spec}\left(\Bbb{B}\right) : a \notin I\}\]
forms a base for the open sets of the Stone topology. It is well-known that $\text{Spec}\left(\Bbb{B}\right)$, equipped with the Stone topology, is a compact Hausdorff space.

Let $X$ be a non empty set. A family $\mathcal{A}\subseteq \mathcal{P}(X)$ is called a  {\textit{\bfseries  field of subsets of $X$}} if 
\begin{itemize}
\item  $A\cup B\in \mathcal{A}$, for each $A, B\in \mathcal{A}$.
\item   $X\setminus A\in \mathcal{A}$, if $A\in \mathcal{A}$.
\end{itemize}
 Every field of sets is a Boolean algebra. For an infinite cardinal number $\kappa$, a field $\mathcal{A}$ is called {\textit{\bfseries $\kappa^+$-additive}} if for each family $\{A_{i}: i\in \Bbb{I}\}$ of elements in $\mathcal{A}$ with $|I|\leq \kappa$ the union $\bigcup_{i\in \Bbb{I}}A_{i}$ belongs to $\mathcal{A}$, i.e., $\mathcal{A}$ is closed under arbitrary $\kappa$-intersections. $\aleph_{0}^+$-additive field of sets is called a \textit{$\sigma$-field}.

Let $X$ be a topological space. We denote by $\EuScript{C}\EuScript{O}(X)$  the set of all clopen subsets of $X$. It is well-known that $\EuScript{C}\EuScript{O}(X)$ is closed under finite unions an complementation. Hence  $\EuScript{C}\EuScript{O}(X)$  turns out to be a field of subsets of $X$. The following theorem on representation of a Boolean algebra is well-known.

\begin{theorem}\label{1.2.1}(Stone representation theorem)
\begin{enumerate}
\item  If $\Bbb{B}$ is a Boolean algebra, then the Boolean algebras $\Bbb{B}$ and $\EuScript{C}\EuScript{O}\left(\text{Spec}\left(\Bbb{B}\right)\right)$ are isomorphic.
\item If $X$ is a zero-dimensional compact Hausdorff space, then $X$ is homeomorphic to the space $\text{Spec}\left(\EuScript{C}\EuScript{O}(X)\right)$.
\end{enumerate}
\end{theorem}

A topological space is said to be {\textit{\bfseries extremally disconnected}} if for each open set $U\subseteq X$, the closure $\text{cl}_{X}U$ is open in $X$. The importance of such spaces was observed by Stone in the following theorem. For the proof of part (1) see \cite{Wa} and for the proof of part (2) see \cite[Theorem 39]{GH}.

\begin{theorem}\label{1.2.2}(completeness)
\begin{enumerate}
\item Let $X$ be a zero-dimensional topological space. The field $\EuScript{C}\EuScript{O}(X)$  is a complete Boolean algebra if and only if $X$ is an extremally disconnected space.
\item If $\Bbb{B}$ is a Boolean algebra, then $\Bbb{B}$ is a complete Boolean algebra  if and only if  $\text{Spec}\left(\Bbb{B}\right)$  is an extremally disconnected space.
\end{enumerate}
\end{theorem}

We conclude with recalling interplay between complete Boolean algebras and commutative rings  that was discovered by  M. Contessa in \cite{Co}. 

\begin{theorem}\label{1.2.3}
A commutative ring $R$ is a Baer ring if and only if it is a weak Baer ring and $\text{Id}\left(R\right)$ is a complete Boolean algebra.
\end{theorem}

\subsection{Rings of measurable functions, history and outline}
For every non empty set $X$ and every real-valued function $f\in \Bbb{R}^X$ the \textit{zero-set} of $f$ is $f^{\leftarrow}(0)$ and is denoted by $Z(f)$, the \textit{cozero-set} of $f$ is $X\setminus Z(f)$ and is denoted by $\text{coz}(f)$. Let $\mathcal{A}$ be a field of subsets of $X$. A ring of real-valued $\mathcal{A}$-measurable functions is:
\[\EuScript{M}(X,\mathcal{A})=\{f:X\rightarrow \Bbb{R}: f^{\leftarrow}\left(U\right)\in \mathcal{A}\hspace{0.2cm}\text{for every open subset}\hspace{0.1cm} U \hspace{0.1cm}\text{of}\hspace{0.1cm} \Bbb{R}\}\]
with induced pointwise multiplication and addition of $\Bbb{R}^X$.   Hager in \cite{H1} presented the first serious research in the theory of rings of measurable functions. It was continued by  Viertl in \cite{V1, V2, V3, V4}. Later, some works on ring theoretic properties of  $\EuScript{M}(X,\mathcal{A})$ appeared in \cite {AHM, M}. We remark that in all of these references  the class $\mathcal{A}$  is assumed to be a $\sigma$-field. It is remarkable that when  $\mathcal{A}$ is a $\sigma$-field, the ring  $\EuScript{M}(X,\mathcal{A})$ must be uniformly closed, i.e., the uniform limit of a sequence of $\mathcal{A}$- measurable functions is $\mathcal{A}$-measurable. In \cite{AAMS}, eliminating the $\aleph_0^+$-additiveness of the class $\mathcal{A}$,  the concept of a field of sets  was generally assumed and some ring theoretic properties of  $\EuScript{M}(X,\mathcal{A})$  was discussed.  For example, the next result has been proved in \cite{AAMS}.

\begin{lemma}\label{1.3.1} 
If $\mathcal{A}$ is an arbitrary field of subsets of a set $X$, then $\EuScript{M}(X,\mathcal{A})$ is a regular ring. 
\end{lemma}

In connection with self-injectivity of the ring $\EuScript{M}(X,\mathcal{A})$, the following result has been proved in \cite{AHM}.

\begin{theorem}\label{1.3.2}
Let $\mathcal{A}$ be a $\sigma$-field of subsets of a set $X$. If $\mathcal{A}$ is closed under arbitrary union, then $\EuScript{M}(X,\mathcal{A})$ is a self-injective ring.
\end{theorem}
The authors  raised the following as an open question, see \cite[\S 5]{AHM}.

\begin{question}\label{1.3.3}
If $\EuScript{M}(X,\mathcal{A})$ is a self-injective ring (when $\mathcal{A}$ is a $\sigma$-field), then is $\mathcal{A}$ necessarily closed under arbitrary union?
\end{question}
Our main aim in this article is to settle this question. To examine the problem thoroughly we must first subject it to a certain transformation by using some notions from topological spaces and Boolean algebras. Let us give an outline of the content of the rest of this article.  In section 2,  we assume that $\mathcal{A}$ is an arbitrary field of subsets of a set $X$. We study some supplementary results on ideal theory of the ring $\EuScript{M}(X,\mathcal{A})$. We  observe an interplay between ring ideals of  $\EuScript{M}(X,\mathcal{A})$ and order ideals of $\mathcal{A}$. It is shown that the maximal ideal space of the ring $\EuScript{M}(X,\mathcal{A})$ is homeomorphic to the Stone space of the Boolean algebra $\mathcal{A}$.

 Section 3 is devoted to an answer of Question \ref{1.3.3}. Appealing the ring of locally constant continuous real valued functions, we drop the $\aleph_0^+$-additivity of $\mathcal{A}$ in Question \ref{1.3.3} and  provide necessary and sufficient conditions  for the ring  $\EuScript{M}(X,\mathcal{A})$ to be self-injective. It turns out that in  the sense of non measurable cardinals the Question \ref{1.3.3} has an affirmative answer. Section 4  is devoted to self-injectivity of  $\EuScript{M}(X,\mathcal{A})$ modulo its socle. In this section we assume that $\mathcal{A}$ is a $\sigma$-field. We  see that  in contrast to $\EuScript{M}(X,\mathcal{A})$, whenever $X$ has a non measurable cardinal and $\mathcal{A}$ has an infinite number of atoms, despite the factor ring  $\EuScript{M}(X,\mathcal{A})/\text{Soc}\left(\EuScript{M}(X,\mathcal{A})\right)$ is always $\aleph_0$-self-injective, it  is never a self-injective ring.

\section{Ring ideals of  $\EuScript{M}(X,\mathcal{A})$ via order ideals of $\mathcal{A}$}

Throughout this section, $\mathcal{A}$ will denote an arbitrary field of subsets of a set $X$. We will study the interactions between   some  order theoretic  properties of  the Boolean algebra $\mathcal{A}$ and some ring theoretic properties of  $\EuScript{M}(X,\mathcal{A})$. Every ring ideal of  $\EuScript{M}(X,\mathcal{A})$ have a nice characterization through order ideals of $\mathcal{A}$.

 To each  $f\in \EuScript{M}(X,\mathcal{A})$,  we associate a real valued function $f^*$  as follows:
\[f^*(x)=\left\{\!\!\!\begin{array}{ll}
\frac{1}{f(x)} & x\in \text{coz}(f),\\
1 & x\in Z(f).\\
\end{array} \right.
\]
Clearly, $f^*$ is a unit of $\EuScript{M}(X,\mathcal{A})$ and $ff^*=\chi_{_{\text{coz}(f)}}$. 

\begin{definition}
Let $\mathcal{A}$ be a field of subsets of a set $X$. If $I$ is an order ideal of $\mathcal{A}$, then 
\[\overline{I}=\{f\in \EuScript{M}(X,\mathcal{A}): \text{coz}(f)\in I\}.\] 
\end{definition}

 The following proposition shows a reciprocal lattice isomorphism between the order ideals of $\mathcal{A}$ and the ring ideals of $\EuScript{M}(X,\mathcal{A})$.

\begin{proposition}\label{2.1}
Let $\mathcal{A}$ be a field of subsets of a set $X$. 
\begin{enumerate}
\item if $I$ is an order ideal of $\mathcal{A}$, then $\overline{I}$ is a ring ideal of $\EuScript{M}(X,\mathcal{A})$.
\item if $J$ is a ring ideal of $\EuScript{M}(X,\mathcal{A})$, there is a unique  order ideal $\Theta(J)$ of $\mathcal{A}$ such that $J=\overline{\Theta(J)}$.
\item if $I$ and $K$ are two order ideals of $\mathcal{A}$, then $I\subseteq K$ if and only if $\overline{I}\subseteq \overline{K}$.
\item an ideal $M$ of $\EuScript{M}(X,\mathcal{A})$ is a maximal ring ideal if and only if $\Theta(M)$ is a maximal order ideal.
\item an ideal $J$ of $\EuScript{M}(X,\mathcal{A})$ is a minimal ring ideal if and only if $\Theta(J)$ is a minimal order ideal, i.e., there is $A\in \text{At}\left(\mathcal{A}\right)$  such that $\Theta(J)=\{0, A\}$.
\end{enumerate}
\end{proposition}

\begin{proof}
We only prove statements (1) and (2). The proofs of the statements  (3), (4) and (5) are evident.

Proof of  (1).  Suppose that $f,g\in\overline{I}$. Since $\text{coz}(f+g)\subseteq \text{coz}(f)\cup \text{coz}(g)$ and $I$ is an order  ideal, we observe that $\text{coz}(f+g)\in I$ and hence $f+g\in \overline{I}$. For every $f\in \EuScript{M}(X,\mathcal{A})$ and $g\in \overline{I}$, the inclusion $\text{coz}(fg)\subseteq \text{coz}(f)\cap \text{coz}(g)\subseteq \text{coz}(g)$ implies that $\text{coz}(fg)\in I$ and hence $fg\in \overline{I}$.

Proof of (2). Assume that $J$ is a ring ideal of $\EuScript{M}(X,\mathcal{A})$. Define 
\[\Theta(J):=\{\text{coz}(f): f\in J\}.\]
For each $f,g\in J$, since $\text{coz}(f)\cup \text{coz}(g)=\text{coz}(f^2+g^2)$ and $f^2+g^2\in J$, we obtain that $\Theta(J)$ is closed under finite union. Suppose that $A\in \mathcal{A}$ and $f\in J$ is such that $A\subseteq \text{coz}(f)$. Since $ff^{*}\chi_{_{A}}=\chi_{_{A}}$, we have $\chi_{_{A}}\in J$ and hence $A\in\Theta(J)$. The uniqueness of the ideal $\Theta(J)$ is clear by definition.
\end{proof}

With the aid of Proposition \ref{2.1}, we can give a complete description of the maximal ideal space of  $\EuScript{M}(X,\mathcal{A})$. We will see that the structure space of $\EuScript{M}(X,\mathcal{A})$ is homeomorphic to the Stone space of the Boolean algebra $\mathcal{A}$. To keep the article self-contained, we start by recalling the  hull-kernel topology  for commutative rings.

Recall that if  $R$ is a commutative ring, we denote the collection of  maximal ideals of $R$ by  $\text{Max}(R)$. 

$\text{Max}(R)$ can be equipped with the hull-kernel topology. Recall that the collection of all sets of the form
\[\Bbb{D}(a)=\{\mathfrak{m}\in \text{Max}(R): a\notin \mathfrak{m}\}\]
where $a\in R$, forms a basis for the open sets of the hull-kernel topology.  

If $R$ is a commutative ring, let $\text{Id}(R)$ be the Boolean algebra of idempotents of $R$, where the Boolean
algebra operations are given by
\[e\wedge f=ef\qquad e\vee f= e+f-ef\qquad e^{\prime}=1-e.\]

The following theorem shows that every idempotent  in a factor ring of $\EuScript{M}(X,\mathcal{A})$ can be lifted to an idempotent of $\EuScript{M}(X,\mathcal{A})$.

\begin{proposition}\label{2.2}
Let $I$ be a lattice ideal of $\mathcal{A}$. For each idempotent $f+\overline{I}$ in $\EuScript{M}(X,\mathcal{A})/\overline{I}$, there is an idempotent $e\in \EuScript{M}(X,\mathcal{A})$ such that $f+\overline{I}=e+\overline{I}$.
\end{proposition}

\begin{proof}
Assume that $f+\overline{I}\in \text{Id}\left(\EuScript{M}(X,\mathcal{A})/\overline{I}\right)$. Thus $f^2-f\in \overline{I}$. Furthermore, there are $A\in I$ and $h\in \EuScript{M}(X,\mathcal{A})/\overline{I}$ with $f^2-f=h \chi{_{A}}$. For each $x\in X\setminus A, f^2(x)=f(x)$. Now define $e_f$ as follows:

\[e_f(x)=\left\{\!\!\!\begin{array}{ll}
f(x) & x\in X\setminus A,\\
0 & x\in A.\\
\end{array} \right.
\]
Observe that $e_f^2=e_f$ and $\text{coz}(f-e_f)\subseteq A$. Since $A\in I$, it follows that $f-e_f\in \overline{I}$. This finishes the proof of the proposition.
\end{proof}

\begin{corollary}\label{2.3}
Let $I$ be an order ideal of $\mathcal{A}$. The Boolean algebra $\text{Id}\left(\EuScript{M}(X,\mathcal{A})/\overline{I}\right)$ is isomorphic to the quotient Boolean algebra $\mathcal{A}/I$.
\end{corollary}

\begin{proof}
By Proposition \ref{2.2}, for each idempotent $f+\overline{I}\in \text{Id}\left(\EuScript{M}(X,\mathcal{A})/\overline{I}\right)$
there is $A_f \in \mathcal{A}$ such that $f+\overline{I}=\chi_{_{A_f}}+\overline{I}$ . Define a map  $\lambda$ as follows:

\[
\begin{CD}
\lambda: \text{Id}\left(\EuScript{M}(X,\mathcal{A})/\overline{I}\right)@>>>\mathcal{A}/I\\
f+\overline{I}\!\!\! \! \! \! \! \! \! \! \! \! \! \! \!  @>>>  A_f+I
\end{CD}
\]
$\lambda$ is well-defined, for  if $A,B\in \mathcal{A}$, with $\chi_{_{A}}+\overline{I}=\chi_{_{B}}+\overline{I}$, then  $\text{coz}(\chi_{_{A}}-\chi_{_{B}})=A\oplus B$ belongs to $I$. This means that $A+I=B+I$. Now assume that $A, B\in \mathcal{A}$ and $\lambda(A)=\lambda(B)$. Consequently, $A\oplus B\in I$. Since $A\oplus B=\text{coz}\left(\chi_{_{A}}-\chi_{_{B}}\right)$, it follows that $\chi_{_{A}}+\overline{I}=\chi_{_{B}}+\overline{I}$. Therefore $\lambda$ is one-one. Clearly $\lambda$ is  onto. The map $\lambda$ preserves join, for, 

\[
\begin{split}
\lambda\left(\left(\chi_{_{A}}+\overline{I}\right)\vee\left(\chi_{_{B}}+\overline{I}\right)\right)&=\lambda\left(\chi_{_{A}}+\chi_{_{B}}-\chi_{_{A\cap B}}+\overline{I}\right)\\&=\lambda\left(\chi_{_{A\cup B}}+\overline{I}\right)=A\cup B+I\\&= \lambda \left(\chi_{_{A}}+\overline{I}\right)\vee \lambda\left(\chi_{_{B}}+\overline{I}\right).
\end{split}
\]

The map $\lambda$  preserves complement, for, 
\[
\begin{split}
\lambda\left(\left(\chi_{_{A}}+\overline{I}\right)^{\prime}\right)&=\lambda\left(\left(1-\chi_{_{A}}\right)+\overline{I}\right)\\&=\lambda\left(\chi_{_{A^c}}+\overline{I}\right)=A^c+I\\&= \lambda \left(\chi_{_{A}}+\overline{I}\right)^{\prime}.
\end{split}
\]
Thus $\lambda$ is a Boolean isomorphism.
\end{proof}

\begin{proposition}\label{2.4}
Let $I$ be an order ideal of a field $\mathcal{A}$. Then $\text{Max}\left(\EuScript{M}(X,\mathcal{A})/\overline{I}\right)$ is homeomorphic to the Stone space $\text{Spec}\left(\mathcal{A}/I\right)$.
\end{proposition}
\begin{proof}
For each maximal ring ideal $M/\overline{I}$ of $\text{Max}\left(\EuScript{M}(X,\mathcal{A})/\overline{I}\right)$, by Proposition \ref{2.1}, there is a maximal order ideal $\Theta(M)/I$ of $\mathcal{A}/I$  such that $M=\overline{\Theta(M)}$ and $I\subseteq \Theta(M)$. Define the map $\varphi: \text{Max}\left(\EuScript{M}(X,\mathcal{A})/\overline{I}\right)\rightarrow \mathcal{A}/I$ by the rule  $\varphi(M/\overline{I})=\Theta (M)/I$. Since for each $\overline{f}=f+\overline{I}\in \EuScript{M}(X,\mathcal{A})/\overline{I}$,
\[\varphi\left(\Bbb{D}(\overline{f})\right)=\mathcal{U}\left(\text{coz}(f)+I\right),\] 
and for each $A\in \mathcal{A}$,
\[\varphi^{\leftarrow}\left(\mathcal{U}\left(A+I\right)\right)=\Bbb{D}\left(\chi_{_{A}}+\overline{I}\right),\]
we infer that $\varphi$ is continuous and open. Hence $\varphi$ is a homeomorphism.
\end{proof}
 
If we consider the order ideal $\{\emptyset\}$ of $\mathcal{A}$, as a consequence of Proposition \ref {2.4}, the following corollary is immediate.
\begin{corollary}\label{2.5}
Let $\mathcal{A}$ be a field of subsets of a set $X$. Then the space  $\text{Max}\left(\EuScript{M}(X,\mathcal{A})\right)$ is homeomorphic to the Stone space $\text{Spec}\left(\mathcal{A}\right)$.
\end{corollary}

\begin{example}\label{2.6}
The finite-cofinite algebra of a set $X$ is defined to be 
\[\mathfrak{F}\mathfrak{C}(X)=\{A\subseteq X: A \hspace{0.25cm}\text{is finite or}\hspace{0.25cm} X\setminus A \hspace{0.25cm}\text{is finite}\}.\]
It is well-known that $\text{Spec}\left(\mathfrak{F}\mathfrak{C}(X)\right)$ is homeomorphic to the one-point compactification of the discrete space $X$, i.e., $Y=X\cup\{\infty\}$. By Corollary \ref{2.5}, $\text{Max}\left(\EuScript{M}(X,\mathfrak{F}\mathfrak{C}(X))\right)$ is homeomorphic to the space $Y=X\cup\{\infty\}$ as well.
\end{example}

\begin{example}\label{2.7}
Let $X$ be a zero-dimensional Hausdorff space (i.e., a Hausdorff topological space with a  basis consisting of clopen sets) and $\EuScript{C}\EuScript{O}(X)$ be the field of all clopen subsets of $X$. The Banaschewski compactification of a  zero-dimensional Hausdorff space  is a compact  Hausdorff space $\beta_{0}X$ which contains $X$ as a dense subspace and each continuous real valued function $f:X\rightarrow \Bbb{R}$ with a finite image has an extension to $\beta_0X$; see, e.g. \cite{AKKO}. Banaschewski in \cite{B1}, showed that $\text{Spec}\left(\EuScript{C}\EuScript{O}(X)\right)$ is homeomorphic to $\beta_0X$, the Banaschewski compactification of $X$. According to Corrolary \ref{2.5},  $\text{Max}\left(\EuScript{M}(X,\EuScript{C}\EuScript{O}(X)\right)$ is homeomorphic to $\beta_0X$.
\end{example}

We use the rest of this section to discuss some results about annihilator ideals of the ring $\EuScript{M}(X,\mathcal{A})$. First, we observe a description of annihilator ideals of $\EuScript{M}(X,\mathcal{A})$ via order ideals of $\mathcal{A}$. We recall that for every order ideal $I$ of   a Boolean algebra $\Bbb{B}$, 
\[I^\perp :=\{t\in \Bbb{B}: t\wedge a=0,  \text{for all}\hspace{0.25cm} a\in I\}.\]

\begin{lemma}\label{2.7.5}
Let $I$ be an order ideal of $\mathcal{A}$. Then $\text{Ann}\left(\overline{I}\right)=\overline{I^\perp}$.
\end{lemma}

\begin{proof}
According to Proposition \ref{2.1}, there is a lattice ideal $K$ in $\mathcal{A}$ such that $\text{Ann}\left(\overline{I}\right)=\overline{K}$. It is to be proved that $K=I^\perp$. If $A\in K$, then $\chi_{_{A}}$ belongs to $\text{Ann}\left(\overline{I}\right)$. Hence for each $E$ in $I$, $\chi_{_{A}}\chi_{_{E}}=0$ and thus $A\cap E=\emptyset$. This means that  $A\in I^\perp$. Conversely, suppose that $B\in I^\perp$. For each $f\in \overline{I}$, we have $\text{coz}(f)\in I$, and hence $\text{coz}(f)\cap B=\emptyset$. This implies that $f^*f \chi_{_{B}}=0$. Sinec $f^*$ is a unit, $f\chi_{_{B}}=0$,  hence $\chi_{_{B}}\in \overline{K}$. Therefore $B\in K$.
\end{proof}

Of course, every regular ring is a weak Baer ring. Now, using Corollary \ref{2.3}, Proposition \ref{2.4} and  part (2) of Theorem \ref{1.2.2}, the following result can be recorded as a corollary for Theorem \ref{1.2.3}. 

\begin{corollary}\label{2.9}
Let $I$ be an order ideal of $\mathcal{A}$. The following statements are equivalent:
\begin{enumerate}
\item $\EuScript{M}(X,\mathcal{A})/\overline{I}$ is a Baer ring.
\item $\mathcal{A}/I$ is a complete Boolean algebra. 
\item $\text{Spec}\left(\mathcal{A}/I\right)$ is extremally disconnected.
\end{enumerate}
\end{corollary}

If $I$ is the null ideal of $\mathcal{A}$, i.e., $I=\{\emptyset\}$, the following corollary is of independent interest.

\begin{corollary}\label{2.10}
$\EuScript{M}(X,\mathcal{A})$ is a Baer ring if and only if $\mathcal{A}$ is a complete Boolean algebra.
\end{corollary}

In the following example, we observe a ring of measurable functions which is not a Baer ring  but  does have an ideal whose  factor ring  is a Baer ring.

\begin{example}\label{2.11}
Starting with a compact Hausdorff space $X$, one forms $\mathcal{B}_X$, the $\sigma$-field of Borel sets of $X$ (i.e., the smallest $\sigma$-field  over $X$ which contains all open sets.) Let $I_m$ be  the ideal of meager Borel sets. It is well-known that $\mathcal{B}_X/I_m$ is isomorphic to $\EuScript{R}\EuScript{O}(X)$, the complete Boolean algebra of all regular open sets in $X$. (see e.g. \cite[Theorem 29]{GH}). Thus,  Corollary \ref{2.9} applies, and hence the factor ring $\EuScript{M}(X,\mathcal{B}_X)/\overline{I_m}$ is a Baer ring. Note that  in general, $\mathcal{B}_X$  is  not complete, and hence $\EuScript{M}(X,\mathcal{B}_X)$  is not a Baer ring.
\end{example}

\section{Self-injectivity of $\EuScript{M}(X,\mathcal{A})$ }

\begin{definition}\label{3.1}
Let $X$ be a zero-dimensional Hausdorff topological space. A real-valued function $f:X\rightarrow \Bbb{R}$ is locally constant if for each $t\in X$, there is a neighborhood $N$ of $t$ such that $f(s)=f(t)$ for all $s\in N$. If $f:X\rightarrow \Bbb{R}$  is continuous, it is clear that $f$ is  locally constant if and only if for each $r\in \Bbb{R}$, the inverse image $f^{\leftarrow}(r)$ is clopen in $X$. The set of all continuous and  locally constant real valued functions is denoted by $C(X,\Bbb{R}_{\text{disc}})$.
\end{definition}

The following proposition shows that when $X$ is a zero-dimensional Hausdorff  space, then the ring $\EuScript{M}(X,\EuScript{C}\EuScript{O}(X))$ is identical to $C(X,\Bbb{R}_{\text{disc}})$. First, we remark that  the usual Euclidean topology on the set $\Bbb{R}$ of real numbers  is denoted by $\Bbb{R}_{e}$ and $\Bbb{R}$, endowed
 with the discrete topology,  is denoted by  $\Bbb{R}_{\text{disc}}$.

\begin{proposition}\label{3.2}
Let  $X$ is a zero-dimensional Hausdorff  space. Then $\EuScript{M}(X,\EuScript{C}\EuScript{O}(X))=C(X,\Bbb{R}_{\text{disc}})$.
\end{proposition}

\begin{proof}
Let $f\in \EuScript{M}(X,\EuScript{C}\EuScript{O}(X))$. For any open subset $U$ of  $\Bbb{R}_{e}$, the inverse image $f^{\leftarrow}\left(U\right)$ is a clopen subset of $X$. Thus, for each $r\in \Bbb{R}$, Both sets $f^{\leftarrow}\left((-\infty, r)\right)$ and $f^{\leftarrow}\left((r, \infty)\right)$ are clopen in $X$. Taking complement, Both sets $f^{\leftarrow}\left([r, \infty)\right)$ and $f^{\leftarrow}\left((-\infty, r]\right)$ belong to $\EuScript{C}\EuScript{O}(X)$. Since $\EuScript{C}\EuScript{O}(X)$ is closed under finite intersections, it follows that $f^{\leftarrow}\left(r\right)\in \EuScript{C}\EuScript{O}(X)$. Whence $f\in C(X,\Bbb{R}_{\text{disc}})$. Conversely, given $h\in C(X,\Bbb{R}_{\text{disc}})$, if $U$ is an open subset of $\Bbb{R}_{e}$,	taking  into consideration that $U$ be clopen in $\Bbb{R}_{\text{disc}}$ and $f$ be continuous, the inverse image $f^{\leftarrow}\left(U\right)$ belongs to $\EuScript{C}\EuScript{O}(X)$. Thus, $f$ turns out to be in $\EuScript{M}(X,\EuScript{C}\EuScript{O}(X))$.
\end{proof}

\begin{definition}\label{3.3}
Let $\mathcal{A}$ be an arbitrary field of subsets of a set $X$. $\mathcal{A}$ is a base for a topology on $X$. The set $X$ equipped with this topology is denoted by $X_{\mathcal{A}}$. Since for each $A\in \mathcal{A}$, we have $X\setminus A\in \mathcal{A}$, the topological  space $X_{\mathcal{A}}$ is zero-dimensional. 
\end{definition}

 In the light of Proposition \ref{3.2}, for every field of subsets of a set $X$, we can consider every ring $\EuScript{M}(X,\mathcal{A})$ as a subring of $ C(X_{\mathcal{A}},\Bbb{R}_{\text{disc}})$. Indeed, the algebra $ C(X_{\mathcal{A}},\Bbb{R}_{\text{disc}})$ turns out to be   a ring of quotients of  $\EuScript{M}(X,\mathcal{A})$.

\begin{proposition}\label{3.4}
Let $\mathcal{A}$ be an arbitrary field of subsets of a set $X$. The algebra $ C(X_{\mathcal{A}},\Bbb{R}_{\text{disc}})$  is a ring of quotients of  $\EuScript{M}(X,\mathcal{A})$.
\end{proposition}

\begin{proof}
Since $\mathcal{A}\subseteq \EuScript{C}\EuScript{O}\left(X_{\mathcal{A}}\right)$, it is clear that  $\EuScript{M}(X,\mathcal{A})$ is a subring of $ \EuScript{M}(X,\EuScript{C}\EuScript{O}\left(X_{\mathcal{A}}\right))= C(X_{\mathcal{A}},\Bbb{R}_{\text{disc}})$. For each nonzero element $f$ in $ C(X_{\mathcal{A}},\Bbb{R}_{\text{disc}})$, there exists a nonzero real number $r\in \Bbb{R}$ for which $f^{\leftarrow}(r)$ is nonempty and clopen in $X_{\mathcal{A}}$. Since  $\mathcal{A}$ is a base for the topology of $ X_{\mathcal{A}}$, there is a nonempty $A\in \mathcal{A}$ such that $A\subseteq f^{\leftarrow}(r)$. Hence we observe that $\chi_{_{A}}f=r\chi_{_{A}}$ is a nonzero element of $\EuScript{M}(X,\mathcal{A})$.
\end{proof}

Now let us recall that for a  zero-dimensional (not necessarily Hausdorff topological) space $X$, if $Y$ is a dense subset of $X$, then the restriction homomorphism $f\rightarrow f|_{Y}$ from $C(X,\Bbb{R}_{\text{disc}})$ into $C(Y,\Bbb{R}_{\text{disc}})$ is a monomorphism. We can consider $C(Y,\Bbb{R}_{\text{disc}})$ as an over ring of $C(X,\Bbb{R}_{\text{disc}})$, ( i.e.,  $C(X,\Bbb{R}_{\text{disc}})\subseteq C(Y,\Bbb{R}_{\text{disc}})$). The family of all dense open subsets in $X$ is denoted by $\mathfrak{G}(X)$. Since $\mathfrak{G}(X)$ is closed under finite intersections, we are invite to consider the direct limit ring $\underrightarrow{ \text{lim}}_{_{_{Y\in \mathfrak{G}(X)}}} C(Y,\Bbb{R}_{\text{disc}})$. It is well-known that $\underrightarrow{ \text{lim}}_{_{_{Y\in \mathfrak{G}(X)}}} C(Y,\Bbb{R}_{\text{disc}})$ may be thought of as $\bigcup_{_{Y\in \mathfrak{G}(X)}}  C(Y,\Bbb{R}_{\text{disc}})$, where we identify $f\in C(Y,\Bbb{R}_{\text{disc}})$ with $g\in C(Z,\Bbb{R}_{\text{disc}})$, whenever $f$ and $g$ agree on $Y\cap Z$.

We require the following lemma.  (For the proof  see \cite[section 4.3]{FGL}).

\begin{lemma}\label{3.5}
Let $X$ be an arbitrary topological space. Then 
\[\text{Q}_{\text{max}}\left(C(X,\Bbb{R}_{\text{disc}})\right)=\underrightarrow{ \text{lim}}_{_{_{Y\in \mathfrak{G}(X)}}} C(Y,\Bbb{R}_{\text{disc}}).\]
\end{lemma}

Let us recall that for any commutative semiprime ring $R$, if $S$ is   a ring of quotients of $R$, then $\text{Q}_{\text{max}}\left(R\right)=\text{Q}_{\text{max}}\left(S\right)$. Thus, combining Proposition \ref{3.4} and Lemma \ref{3.5}, we observe the following immediate corollary. 

\begin{corollary}\label{3.6}
Let $X$ be a nonempty set and $\mathcal{A}$ an arbitrary field of subsets of $X$. Then 
\[\text{Q}_{\text{max}}\left(\EuScript{M}(X,\mathcal{A})\right)=\underrightarrow{ \text{lim}}_{_{_{Y\in \mathfrak{G}(X_{\mathcal{A}})}}} C(Y,\Bbb{R}_{\text{disc}}).\]
\end{corollary}

In the next proposition, we investigate the $\aleph_{0}$-self-injectivity of the ring $C(X,\Bbb{R}_{\text{disc}})$. 

\begin{proposition}\label{3.9}
Let $X$ be a zero-dimensional Hausdorff space. The ring $C(X,\Bbb{R}_{\text{disc}})$ is  $\aleph_{0}$-self-injective if and only if $X$ is a $P$-space, i.e., $\EuScript{C}\EuScript{O}(X)$ is a $\sigma$-field.
\end{proposition}

\begin{proof}
Assume first that $X$ is a $P$-space. Then, evidently $C(X,\Bbb{R}_{\text{disc}})=C(X)$ and hence the $\aleph_{0}$-self-injectivity follows from \cite[Theorem 1]{EK}. Conversely, suppose that $C(X,\Bbb{R}_{\text{disc}})$ is $\aleph_{0}$-self-injective. It suffices to establish that for each countable family of pairwise disjoint clopen subsets $\{O_{i}: i\in \Bbb{N}\}$, the union $\bigcup_{i\in \Bbb{N}} O_{i}$ is clopen. For each $i\in \Bbb{N}$, consider $e_i$ as the characteristic function of the set $O_{i}$, i.e., $e_i=\chi_{_{O_{i}}}$. The set $\{\frac{1}{n}e_n: n\in \Bbb{N}\}$ is a countable orthogonal subset of  $C(X,\Bbb{R}_{\text{disc}})$. Appealing to Lemma \ref{3.8}, there is $f\in  C(X,\Bbb{R}_{\text{disc}})$ such that $\frac{1}{n^2}e_{n}f=\frac{1}{n}e_{n}$, for each $n\in \Bbb{N}$. It follows that for each $n\in \Bbb{N}$, $f\left(U_n\right)=\{n\}$. Since $f$ is continuous, it maps the set $\text{cl}_{X}\left(\bigcup_{n\in \Bbb{N}} O_{n}\right)$ into the discrete subset $\Bbb{N}$. If $p\in \text{cl}_{X}\left(\bigcup_{n\in \Bbb{N}} O_{n}\right)\setminus \bigcup_{n\in \Bbb{N}} O_{n}$, then $f(p)=m$, for some $m\in \Bbb{N}$. Define $W=f^{\leftarrow}(m)\setminus O_m$. It is clear that $W$ is a neighborhood of $p$, but $W\cap \left(\bigcup_{n\in \Bbb{N}} O_{n}\right)=\emptyset$, a contradiction.
\end{proof}

\begin{remark}\label{3.10}
In \cite{AHM}, it is shown that for every $\sigma$-field $\mathcal{A}$, the ring $\EuScript{M}(X,\mathcal{A})$  is automatically $\aleph_{0}$-self-injective. In \cite{AAMS} Amini et al. asked whether the  $\aleph_{0}$-self-injectivity of the ring $\EuScript{M}(X,\mathcal{A})$ implies that $\mathcal{A}$ is a $\sigma$-field. Proposition \ref{3.9} gives a partial answer to this question, i.e., the $\aleph_{0}$-self-injectivity of the ring  $\EuScript{M}(X,\EuScript{C}\EuScript{O}(X))$ implies that $\EuScript{C}\EuScript{O}(X)$ is a $\sigma$-field.
\end{remark}

 For the proof of the following result  see \cite[Lemma 3.12]{KO}. 

\begin{lemma}\label{3.11}
Assume that $\kappa$ is an infinite cardinal number. Let $\Bbb{B}$ be an $\alpha$-complete Boolean algebra, for every $\alpha<\kappa$. Then for every family $\{a_{\alpha}: \alpha < \kappa\}$ in $\Bbb{B}$, there is a family $\{b_{\alpha}: \alpha < \kappa\}$ consisting of pairwise disjoint elements such that for each $\alpha<\kappa$ , $b_{\alpha}\leq a_{\alpha}$ and \[\text{sup}\left\{ a_{\alpha}: \alpha < \kappa\right\}= \text{sup}\left\{ b_{\alpha}: \alpha < \kappa\right\},\]
if one of these supremum exists.
\end{lemma}

In the next result, we generalize Proposition \ref{3.9} and give an equivalent condition for self-injectivity of the ring $C(X,\Bbb{R}_{\text{disc}})$. First, we remind the reader that a topological space $X$ is a $P_{{\mathfrak{c}}^+}$-space if for any collection $\{C_i : i\in \Bbb{I}\}$ of clopen subsets of $X$ with $|\Bbb{I}|\leq 2^{\aleph_0}$, the intersection $\bigcap_{i\in \Bbb{I}}C_i$ is also clopen.

\begin{theorem}\label{3.12}
Let $X$ be a zero-dimensional Hausdorff space. The ring $C(X,\Bbb{R}_{\text{disc}})$ is self-injective if and only if $X$ is an extremally disconnected $P_{{\mathfrak{c}}^+}$-space.
\end{theorem}

\begin{proof}
Assume that $X$ is an extremally disconnected $P_{{\mathfrak{c}}^+}$-space. Then $C(X,\Bbb{R}_{\text{disc}})=C(X)$ and the self-injectivity of the ring is implied by \cite{AK}. Conversely, assume that $C(X,\Bbb{R}_{\text{disc}})$ is self-injective. First, we show that $X$ is extermally disconnected. The self-injectivity of  $C(X,\Bbb{R}_{\text{disc}})= \EuScript{M}(X,\EuScript{C}\EuScript{O}(X))$ implies that it is a Baer ring. Using Corollary \ref{2.10}, the field $\EuScript{C}\EuScript{O}(X)$ is complete and hence part (1) of Theorem \ref{1.2.2} implies that  $X$ is extermally disconnected. Now we are going to prove that $X$ is a  $P_{{\mathfrak{c}}^+}$-space. Appealing Lemma \ref{3.11}, it is enough  to prove that the union of every at most $2^{\aleph_0}$ pairwise disjoint clopen sets is closed. Suppose that $\{O_i : i\in \Bbb{I}\}$ is a family of pairwise disjoint clopen sets of $X$ with $|\Bbb{I}|\leq 2^{\aleph_0}$. Choose a subset $\{r_i : i\in \Bbb{I}\}$ of positive elements of $\Bbb{R}$ and consider the orthogonal  family $\{\frac{1}{r_i}\chi_{_{O_i}}: i\in \Bbb{I}\}$. Using Lemma \ref{3.8}, the self-injectivity of  $C(X,\Bbb{R}_{\text{disc}})$ implies that there is $f\in C(X,\Bbb{R}_{\text{disc}})$ such that $\frac{1}{r_i^2}\chi_{O_i}f=\frac{1}{r_i}\chi_{O_i}$,  for each $i\in \Bbb{I}$. Whence $f\left(O_i\right)=\{r_i\}$, for each $i\in \Bbb{I}$. Recall that $\{r_i: i\in \Bbb{I}\}$ is a discrete subset of $\Bbb{R}_{\text{disc}}$ and $f$ is a continuous map. Thus $f$ maps $\text{cl}_{X}\left(\bigcup_{i\in \Bbb{I}} O_{i}\right)$ into the  subset  $\{r_i : i\in \Bbb{I}\}$. If $p\in \text{cl}_{X}\left(\bigcup_{n\in \Bbb{I}} O_{i}\right)\setminus \bigcup_{i\in \Bbb{N}} O_{i}$, there is $j\in \Bbb{I}$ such that $f(p)=r_j$.  Define $V=f^{\leftarrow}(r_j)\setminus O_j$. The set $V$ is a  neighborhood of $p$. But $V\cap \left(\bigcup_{i\in \Bbb{I}} O_{i}\right)=\emptyset$, a contradiction.
\end{proof}

Before proving our main result of this section, we need the following proposition. we recall that a field $\mathcal{A}$ of subsets of a set $X$ is {\textit{\bfseries  reduced}} if for every two distinct points $x,y \in X$, there is $A\in \mathcal{A}$ such that $x\in A$ and $y\notin A$. Note that $\mathcal{A}$ is a reduced field of sets if and only if $X_{\mathcal{A}}$ is a Hausdorff space. Sikorski in \cite[\S 7]{Si} showed that every field of sets is Boolean  isomorphic to a reduced field of sets. In the following we extend the result to be suitable for our purpose.

\begin{proposition}\label{3.13}
Let $\mathcal{A}$ be an arbitrary field of subsets of a set $X$. There is a set $Y$, a reduced field $\mathcal{B}$ of  subset of $Y$ and an onto map $\pi: X\rightarrow Y$, such that 
\begin{enumerate}
\item The map $\varphi: \mathcal{B}\rightarrow \mathcal{A}$ defined by $\varphi(B)=\pi^{\leftarrow}(B)$ is an isomorphism between field of sets.
\item The two rings $\EuScript{M}(X,\mathcal{A})$ and $\EuScript{M}(Y,\mathcal{B})$ are isomorphic.
\end{enumerate}
\end{proposition}

\begin{proof}
(a) We define $x\sim y$ in $X$ if $f(x)=f(y)$  for all  $f\in \EuScript{M}(X,\mathcal{A})$. Clearly, this is an equivalence relation on $X$. Put $Y=\{[x]: x\in X\}$, where $[x]$ is the equivalence class of $x\in X$. We define $\pi: X\rightarrow Y$ by $\pi(x)=[x]$. Now for each $f\in \EuScript{M}(X,\mathcal{A})$, let us define $\overline{f}:Y\rightarrow \Bbb{R}$ by $\overline{f}([x])=f(x)$, hence $f=\overline{f}\circ \pi$. Now put \[\mathcal{B}=\{A\subseteq Y: \pi ^{\leftarrow}(A)\in \mathcal{A}\}.\]
Clearly $\mathcal{B}$ is a field of subsets of $Y$. Since for each $f\in \EuScript{M}(X,\mathcal{A})$ and each open set $U\subseteq \Bbb{R}$, we have \[f^{\leftarrow}\left(U\right)=\pi^{\leftarrow}\left(\overline{f}^{^{\leftarrow}}(U)\right),\]
we infer that $\overline{f}\in \EuScript{M}(Y,\mathcal{B})$. It is easy to see that for each $g\in \EuScript{M}(Y,\mathcal{B})$, we have $g\circ \pi \in \EuScript{M}(X,\mathcal{A})$. We also note that for $[x]\neq [y]$ in $Y$, there is $f\in \EuScript{M}(X,\mathcal{A})$ with $f(x)\neq f(y)$. If we put $r=f(x)$, then $f^{\leftarrow}(r)\in \mathcal{A}$ and hence $[x]\in \overline{f}^{^{\leftarrow}}(r)\in \mathcal{B}$ and $[y]\in Y\setminus \overline{f}^{^{\leftarrow}}(r)\in \mathcal{B}$. Therefore $\mathcal{B}$ is a reduced field of sets on $Y$.

(b) Let us define $\Phi: \EuScript{M}(X,\mathcal{A})\rightarrow \EuScript{M}(Y,\mathcal{B})$, by $\Phi(g)=g\circ \pi$. By the proof of part (a) it is then routine to see $\Phi$ is a ring isomorphism.
\end{proof}

The next  corollary collects several useful  consequences of Proposition \ref{3.13}. 
\begin{corollary}\label{3.14}
Let $\mathcal{A}$ be a field of subsets of of a set $X$ and the set $Y$, the field $\mathcal{B}$  and the Boolean isomorphism $\varphi$ be defined as in Proposition \ref{3.13}. The following statements hold:
\begin{enumerate}
\item For an arbitrary infinite cardinal $\kappa$, the field $\mathcal{A}$ is $\kappa^+$-additive if and only if $\mathcal{B}$ is $\kappa^+$-additive.

\item The field $\mathcal{A}$ is $\kappa^+$-complete if and only if $\mathcal{B}$  is $\kappa^+$-additive.

\item An element $A\in \mathcal{A}$ is an atom if and only if $\varphi^{-1}(A)$ is a singleton.
\end{enumerate}
\end{corollary}

\begin{proof}
\begin{enumerate}
\item  Note that the inverse image operator  preserves arbitrary unions. Hence the proof is patent.
\item Since each Boolean isomorphism preserves all suprema, the proof of it clearly holds.
\item Suppose that $A$ is an atom in $\mathcal{A}$. If $t\neq s$ are two elements of the set $\varphi^{-1} (A)$, there is $B\in \mathcal{B}$ such that $t\in \mathcal{B}$ and $s\notin \mathcal{B}$. Hence $\varphi^{-1} (A)\cap \mathcal{B}$ is a proper subset of $\varphi^{-1} (A)$, a contradiction. The converse is patent.
\end{enumerate}
\end{proof}

Now we are ready to present another intrinsic  characterization  of self-injectivity of $ \EuScript{M}(X,\mathcal{A})$. 

\begin{theorem}\label{3.15}
Let $X$ be a nonempty set and $\mathcal{A}$ an arbitrary field of subsets of $X$. The ring $ \EuScript{M}(X,\mathcal{A})$ is self-injective if and only if $\mathcal{A}$ is a complete and $\mathfrak{c}^+$-additive field of sets.
\end{theorem}

\begin{proof}
 By Proposition \ref{3.13}, there is a nonempty set $Y$ and a reduced field $\mathcal{B}$ on $Y$ such that $ \EuScript{M}(X,\mathcal{A})\cong \EuScript{M}(Y,\mathcal{B})$. Note that the topological space $Y_{\mathcal{B}}$ is zero-dimensional and Hausdorff.  For the necessity, since $\EuScript{M}(Y,\mathcal{B})$ is  self-injective, it has no proper ring of quotients. Using Proposition \ref{3.4}, 
\[\EuScript{M}(Y,\mathcal{B})=C(Y_{\mathcal{B}},\Bbb{R}_{\text{disc}})=\EuScript{M}(Y_{\mathcal{B}}, \EuScript{C}\EuScript{O}(Y_{\mathcal{B}})).\]
This  implies that $\mathcal{B}=\EuScript{C}\EuScript{O}(Y_{\mathcal{B}})$. Now by Theorem \ref{3.12}, $Y_{\mathcal{B}}$ is extremally disconnected $P_{\mathfrak{c}^+}$-space. Therefore $\mathcal{B}$ is complete and is closed under arbitrary $\mathfrak{c}$-intersections. Whence applying Corollary \ref{3.14} $\mathcal{A}$ must be  complete and  closed under arbitrary $\mathfrak{c}$-intersections.

For sufficiency, suppose that $\mathcal{A}$ is complete and  closed under taking arbitrary $\mathfrak{c}$-intersections. 
Applying Corollary \ref{3.14}, the field $\mathcal{B}$ on $Y$  is complete and is closed under arbitrary $\mathfrak{c}$-intersections as well. Clearly the space $Y_{\mathcal{B}}$ is Hausdorff and $P_{\mathfrak{c}^+}$-space. We claim that $\mathcal{B}=\EuScript{C}\EuScript{O}(Y_{\mathcal{B}})$. It is enough to show that the field $\mathcal{B}$ contains $\EuScript{C}\EuScript{O}(Y_{\mathcal{B}})$. Let $V\in \EuScript{C}\EuScript{O}(Y_{\mathcal{B}})$. Then there is a family $\{A_i: i\in \Bbb{I}\}$ such that $V=\bigcup_{i\in \Bbb{I}}A_i$. We claim that 
\[A=\text{sup}\{A_i: i\in \Bbb{I}\}=V.\]
Since $V$ is clopen and contains all $A_i$, then $V\subseteq A$. Since $A\setminus V$ is open in $Y_{\mathcal{B}}$, if  $A\setminus V$ were nonempty, there  would exist a nonempty $C\in \mathcal{B}$ such that $C\subseteq A\setminus V$. But the set $A\setminus C$ belongs to $\mathcal{A}$ and it contains all the sets $A_i$, a contradiction. The equality implies extremal disconnectedness of  $Y_{\mathcal{B}}$  and  $\EuScript{M}(Y,\mathcal{B})=C(Y_{\mathcal{B}},\Bbb{R}_{\text{disc}})$. Hence we infer that $\EuScript{M}(Y,\mathcal{B})$ and hence $\EuScript{M}(X,\mathcal{A})$ are self-injective.
\end{proof}

We conclude this section by giving  a positive answer to Question \ref{1.3.3}, subject to a restriction on the cardinality of the set $X$.  We recall that if $x\in X$, the set $\mathcal{I}_x:=\{A\subseteq X: x\notin A\}$ is a proper order ideal of $\mathcal{P}(X)$. Note that for each $x\in X$, the order ideal $\mathcal{I}_x$ is closed under arbitrary countable unions of its elements. Now let $X$ be a set such that each proper order ideal of $\mathcal{P}(X)$ which is closed under countable unions is of this form, then we say that $X$ has a {\textit{\bfseries non measurable cardinal}}; see \cite[\S 26]{Si}

\begin{theorem}\label{3.16}
If $X$ is of non measurable cardinal, then $\EuScript{M}(X,\mathcal{A})$ is self-injective if and only if $\mathcal{A}$ is closed under arbitrary union.
\end{theorem}

\begin{proof}
Using Proposition \ref{3.13}, there is a set $Y$ and a reduced field $\mathcal{B}$ of subsets of $Y$ such that $ \EuScript{M}(X,\mathcal{A})\cong \EuScript{M}(Y,\mathcal{B})$. Since $|Y|\leq |X|$, then $Y$ is also of non measurable cardinal. By Theorem \ref{3.15}, the self-injectivity of $\EuScript{M}(Y,\mathcal{B})$ implies that $\mathcal{B}$ is complete and closed under arbitrary $\mathfrak{c}$-intersections. Therefore $Y_{\mathcal{B}}$ is extremally disconnected $P_{\mathfrak{c}^+}$-space. Since the space $Y_{\mathcal{B}}$  has  a non measurable cardinal, by \cite[Exercise 12H.6]{GJ}, it must be discrete and hence $\mathcal{B}=\mathcal{P}(Y)$. Obviously  $\mathcal{B}$  is closed under arbitrary union. Applying part (1) of Corollary \ref{3.14}, $\mathcal{A}$ must be closed under arbitrary unions. 
\end{proof}

\section{Self-injectivity of $\EuScript{M}(X,\mathcal{A})$  modulo its socle}
For a commutative ring $R$, the socle of $R$, denoted by $\text{Soc}(R)$,  is the sum of all nonzero minimal ideals of $R$. We remind the reader that if $X$ is an arbitrary topological space, $C_F(X)$ (resp., $C^F(X)) $ is the set of all continuous real valued functions with finite support (resp., finite image). It was shown in \cite{KR} that in every Tychonoff space $X$, $\text{Soc}(C(X))$ is equal to $C_F(X)$. For every topological space $X$ we have the following inclusions:
\[C_F(X)\subseteq C^F(X))\subseteq C\left(X,\Bbb{R}_{\text{disc}}\right).\]
Assume that $\mathcal{A}$ is a reduced field of sets of a set $X$. As a consequence of part (5) of Proposition \ref{2.1} and part (3) of Corollary \ref{3.14}, we have the following immediate lemma.

\begin{lemma}\label{4.1}
Let $\mathcal{A}$ be a reduced field of sets of a set $X$. A ring ideal $I$ of $\EuScript{M}(X,\mathcal{A})$ is minimal if and only if it is the principal ring ideal generated by the characteristic function $\chi_{_x}$, where $\{x\}\in \mathcal{A}$.
\end{lemma}

We recall that the set of all isolated points of a topological space $X$ is denoted by $\Bbb{I}(X)$. 

\begin{corollary}\label{4.2}
Let $\mathcal{A}$ be a reduced field of subsets of a set $X$. Then 
\[\text{Soc}(\EuScript{M}(X,\mathcal{A}))=C_F(X_{\mathcal{A}}).\]
\end{corollary}
\begin{proof}
Since $\mathcal{A}$ is a reduced field of subsets of the set $X$, we obtain that $\text{At}\left(\mathcal{A}\right)=\Bbb{I}\left(X_{\mathcal{A}}\right)$. Now the equality follows from the fact that for $x\in X$, the characteristic function $\chi_{_{\{x\}}}$ is continuous just in case $x\in \Bbb{I}\left(X_{\mathcal{A}}\right)$ and just in case
$\chi_{_{\{x\}}}$ belongs to $\EuScript{M}(X,\mathcal{A})$.
\end{proof}

For each $x\in X$ we form a ring ideal of $C(X,\Bbb{R}_{\text{disc}})$: 
\[
M_x^d=\{f\in C\left(X,\Bbb{R}_{\text{disc}}\right): f(x)=0\}.
\]
Indeed the ideal $M_x^d$ is a maximal ring ideal, for the factor ring $C\left(X,\Bbb{R}_{\text{disc}}\right)/ M_x^d$ turns out to be  isomorphic to $\Bbb{R}$. More generally, for a subset $A\subseteq X$, set $M_A^d=\bigcap_{x\in A} M_x^d$. 

We recall that if $X$ is a zero-dimensional Hausdorff space, $\EuScript{C}\EuScript{O}(X)$ is a reduced field of sets and hence by Corollary \ref{4.2}, $C_F(X)= \text{Soc}\left( C(X,\Bbb{R}_{\text{disc}})\right)$. One of our goals is to find a complete characterization  of self-injectivity of the factor ring $C\left(X,\Bbb{R}_{\text{disc}}\right)/ C_F(X)$. First, we need the following results. 

\begin{proposition}\label{4.3}
Let $X$ be a zero-dimensional Hausdorff space. $C_F(X)$ equals $M_{X\setminus \Bbb{I}(X)}^d$ if and only if each infinite subset of isolated points has a limit point in $X$.
\end{proposition}

\begin{proof}
For the necessity, assume that  $C_F(X)=M_{X\setminus \Bbb{I}(X)}^d$. Suppose that $T$ is a subset of $\Bbb{I}(X)$ without any limit point. Hence $T$ is a clopen subset of $X$. Considering the characteristic function $\chi_{_{T}}$, we observe that $\chi_{_{T}}$ belongs to $M_{X\setminus \Bbb{I}(X)}^d$, but $\chi_{_{T}}$ does not belong to  $C_F(X)$, for $\chi_{_{T}}$ has infinite support. 

For the sufficiency, assume that each infinite subset of isolated points has a limit point in $X$. We always have the inclusion $C_F(X)\subseteq M_{X\setminus \Bbb{I}(X)}^d$. If $f\in M_{X\setminus \Bbb{I}(X)}^d$, then $X\setminus Z(f) \subseteq \Bbb{I}(X)$. Since $f\in C\left(X,\Bbb{R}_{\text{disc}}\right)$, the subset  $X\setminus Z(f)$ should be clopen in $X$ and hence it contains no limit points. By our hypothesis, $X\setminus Z(f)$ must be finite. This means that $f\in C_F(X)$.
\end{proof}

\begin{corollary}\label{4.4}
If $X$ is a Stone space, then $C_F(X)=M_{X\setminus \Bbb{I}(X)}^d$.
\end{corollary}

\begin{corollary}\label{4.5}
If $X$ is a zero-dimensional Hausdorff space with the property that each infinite subset of isolated points has a limit point in $X$, then the restriction homomorphism $\phi$ from $C\left(X,\Bbb{R}_{\text{disc}}\right)$ into $C\left (X\setminus \Bbb{I}(X),\Bbb{R}_{\text{disc}}\right)$ has  $C_F (X)$ as its kernel.
\end{corollary}

For a zero-dimensional Hausdorff space $X$, recall that a subset $Y\subseteq X$ is said to be \textit{$C_d$-embedded} in $X$ if for each $f\in C\left(Y,\Bbb{R}_{\text{disc}}\right)$, there is $F\in C\left(X,\Bbb{R}_{\text{disc}}\right)$ with $F|_{Y}=f$. The next lemma guarantees that every compact subset in a zero-dimensional Hausdorff space is $C_d$-embedded, see \cite{E}.

\begin{lemma}\label{4.6}
Let $X$ be a zero-dimensional Hausdorff space and $Y$ be a compact subset of $X$. Then $Y$ is $C_d$-embedded in $X$.
\end{lemma}

Now, we are ready to prove one of our main results.

\begin{theorem}\label{4.7}
Let $X$ be a zero-dimensional Hausdorff space. The factor ring $C\left(X,\Bbb{R}_{\text{disc}}\right)/ C_F(X)$ is self-injective if and only if the following conditions hold:
\begin{enumerate}
\item Each infinite subset of isolated points has a limit point in $X$.
\item $X\setminus \Bbb{I}(X)$ is an extremally disconnected $P_{\mathfrak{c}^+}$-space.
\item $X\setminus \Bbb{I}(X)$ is $C_d$-embedded in $X$.
\end{enumerate}
\end{theorem}

\begin{proof}
Sufficiency: By our hypothesis, since parts (1) and (3) hold, using Corollary \ref{4.5} and the first isomorphism theorem, the factor ring $C\left(X,\Bbb{R}_{\text{disc}}\right)/ C_F(X)$ is isomorphic to the ring $C(X\setminus \Bbb{I}(X),\Bbb{R}_{\text{disc}})$. Since $X\setminus \Bbb{I}(X)$ is an extremally disconnected $P_{\mathfrak{c}^+}$-space, Theorem \ref{3.12} implies that  $C\left(X\setminus \Bbb{I}(X),\Bbb{R}_{\text{disc}}\right)$ is self-injective. Thus the isomorphism shows that the factor ring $C\left(X,\Bbb{R}_{\text{disc}}\right)/ C_F(X)$  turns out to be self-injective.

Necessity: Assume that the factor ring $C\left(X,\Bbb{R}_{\text{disc}}\right)/ C_F(X)$  is self-injective. For part (1), suppose that $S$ is a countably infinite subset of $\Bbb{I}(X)$ without any limit point. It is immediately clear that each subset of $S$ is clopen in $X$. Choose a countable partition $\mathfrak{A}=\{A_i: i\in \Bbb{N}\}$ for $S$, i.e., for each $i\neq j$ in $\Bbb{N}$, $A_i\cap A_j=\emptyset$ and $S=\bigcup_{i\in \Bbb{N}}A_i$. Applying Zorn's lemma, $\mathfrak{A}$ is contained in a family $\mathfrak{K}$ of infinite clopen subsets of $X$ which is maximal with respect to the property that for every distinct $U, V\in \mathfrak{K}$, $U\cap V$ is finite . For each $K\in \mathfrak{K}$, $e_K$ is the characteristic function of $K$. The set 
\[\{e_K+C_F(X): K\in  \mathfrak{K}\}\]
is an orthogonal subset of the factor ring $C\left(X,\Bbb{R}_{\text{disc}}\right)/ C_F(X)$. By virtue of the self-injectivity of the factor ring $C\left(X,\Bbb{R}_{\text{disc}}\right)/ C_F(X)$, there is an element $f\in C\left(X,\Bbb{R}_{\text{disc}}\right)$ such that $f+C_F(X)$ separates the set $\{e_K+C_F(X): K\in \mathfrak{A}\}$ from the set $\{e_K+C_F(X): K\in \mathfrak{K}\setminus \mathfrak{A}\}$, i.e., for each $i\in \Bbb{N}$, 

\begin{align}\label{q1}
fe_{_{A_i}}-e_{_{A_i}}\in C_F(X),
\end{align}
and for each $K\in \mathfrak{K\setminus\mathfrak{A}}$,

\begin{align}\label{q2}
fe_{_{K}}\in C_F(X).
\end{align}

For each $i\in \Bbb{N}$, (\ref{q1}) implies that there is a finite subset $F_i$ in $X$ so that for each $x\in A_i\setminus F_i$,

\begin{align}\label{q3}
f(x)e_{_{A_i}}(x)=e_{_{A_i}}(x)
\end{align}

This means that for each $i\in \Bbb{N}$, $ A_i\cap\left(X\setminus Z(f)\right)\setminus A_i$ is infinite. For each $i\in \Bbb{N}$ take $a_i\in A_i\cap\left(X\setminus Z(f)\right)\setminus A_i$ and consider the set $A=\{a_i: i\in \Bbb{N}\}$. Note that $A$ is a subset of $S$ and by our hypothesis $A$ is clopen in $X$. For each $i\in \Bbb{N}$, $A\cap A_i=\{a_i\}$. For each $K\in \mathfrak{K}\setminus \mathfrak{A}$, by (\ref{q2}), since $A\cap K\subseteq \left(X\setminus Z(f)\right)\cap K$, it  follows that $A\cap K$ must be finite. Thus the family $\EuScript{C}=\mathfrak{K}\cup \{A\}$ is a family of infinite clopen subsets of $X$ whose each pair of distinct elements, have finite intersections, a contradiction. 

Since part (1) holds, it follows that the restriction homomorphism $f\rightarrow f|_{X\setminus \Bbb{I}(X)}$ from $C\left(X,\Bbb{R}_{\text{disc}}\right)$ into $C\left(X\setminus \Bbb{I}(X),\Bbb{R}_{\text{disc}}\right)$ has $C_F(X)$ as its kernel; see Proposition \ref{4.3}.  The first isomorphism theorem implies that there is an isomorphism $\Phi$ from $C\left(X, \Bbb{R}_{\text{disc}}\right)/ C_F(X)$ into  $C\left(X\setminus \Bbb{I}(X),\Bbb{R}_{\text{disc}}\right)$ which sends each element $f+C_F(X)$ to the element $f|_{X\setminus \Bbb{I}(X)}$. We will show that $C\left(X\setminus \Bbb{I}(X),\Bbb{R}_{\text{disc}}\right)$ is a ring of quotients of the image of $\Phi$. Note that
\[\text{Im}(\Phi)=\left\{f\in C\left(X\setminus \Bbb{I}(X),\Bbb{R}_{\text{disc}}\right): \exists F\in C\left(X,\Bbb{R}_{\text{disc}}\right), F|_{X\setminus \Bbb{I}(X)}=f\right\}.\]

To show this claim, let $f$ be a non-zero element of $C\left(X\setminus \Bbb{I}(X),\Bbb{R}_{\text{disc}}\right)$. Take a non-zero real number $r$ for which $f^{\leftarrow}(r)$ is a non empty clopen subset of $X\setminus \Bbb{I}(X)$. For $x\in f^{\leftarrow}(r)$, there is an infinite clopen subset $V\subseteq X$ which contains $x$ and $V\cap \left(X\setminus \Bbb{I}(X)\right)\subseteq f^{\leftarrow}(r)$. Since $V$ is infinite, the characteristic function $\chi_{_{V}}$ does not belong to $C_F(X)$. Hence $\chi_{_{V}}+C_F(X)$ is a non-zero element of the factor ring $C\left(X,\Bbb{R}_{\text{disc}}\right)/ C_F(X)$. Since $\Phi$ is an isomorphism, the element $g=\Phi(\chi_{_{V}}+C_F(X))=\chi_{_{V\setminus\Bbb{I}(X)}}$ is a non-zero element of $C\left(X\setminus \Bbb{I}(X),\Bbb{R}_{\text{disc}}\right)$. Clearly $gf=r\chi_{_{V\setminus\Bbb{I}(X)}}$  is the restriction of the function $r\chi_{_{V}}$. Therefore $gf$ is a non-zero element of $\text{Im}(\Phi)$.
\end{proof}

\begin{remark}\label{4.7.4}
Let us remind the reader that if $X$ is a zero-dimensional Hausdorff space,   $\Bbb{I}(X)=\Bbb{I}(\beta_0 X)$.
There exists a ring isomorphism $\sigma$ from $C^F(X)$ onto $C^F (\beta_0 X)$. (See Example \ref{2.7}).
 Thus the isomorphism $\sigma : C^F(X)\cong C^F(\beta_0 X)$, sends the ideal $C_F(X)$ onto the ideal $ C_F(\beta_0 X)$. Since $\beta_0 X$ is compact
\[C^F(\beta_0 X)= C\left(\beta_0 X,\Bbb{R}_{\text{disc}}\right).\]
Using Corollary \ref{4.4} and Lemma \ref{4.6}  the following isomorphisms hold.
\[\frac{C^F(X)}{C_F(X)}\cong \frac{C^F(\beta_0 X)}{C_F(\beta_0 X)}=\frac{C\left(\beta_0 X,\Bbb{R}_{\text{disc}}\right)}{C_F(\beta_0 X)}\cong C\left(\beta_0 X\setminus \Bbb{I}(X),\Bbb{R}_{\text{disc}}\right)\]
This means that, regarding the aforementioned isomorphism, whenever necessary we may assume that $X=\beta_0 X$
\end{remark}

From Theorem \ref{4.7}  we then get the following result.

\begin{proposition}\label{4.7.5}
Let $X$ be a zero-dimensional Hausdorff space. The factor ring $C^F(X)/C_F(X)$ is  self-injective  if and only if $X$ is compact and $X\setminus\Bbb{I}(X)$ is finite.
\end{proposition}

\begin{proof}
The sufficiency is clear. We only prove the  necessity.  Suppose that the factor ring $C^F(X)/C_F(X)$ is self-injective. By Remark \ref{4.7.4} we can infer that the factor ring $C^F(\beta_0 X)/C_F(\beta_0 X)$ must be self-injective. Part (2) of Theorem \ref{4.7} implies that the compact subset $\beta_0 X\setminus \Bbb{I}(X)$ is an extremally disconnected $P_{\mathfrak{c}^+}$-space. Thus $\beta_0 X\setminus \Bbb{I}(X)$ must be finite. Let $ X = \Bbb{I}(X) \cup A$. Hence no point in $A$ is isolated. Assume that $X$ is not compact, and consider $\beta_0 X$. Pick a point $p$ in $\beta_0 X \setminus X$. Since $\beta_0 X \setminus \Bbb{I}(X)$ is finite, the subspace $A \cup (\beta_0 X\setminus X) $ of $\beta_0 X$ is finite, hence discrete. Hence there is a compact neighborhood $C$ of $p$ in $\beta_0 X$ such that $C \cap (A \cup (\beta_0 X \setminus X)) = \{p\}$. Since $p$ is not isolated in $\beta_0 X$, it follows that $C \cap \Bbb{I}(X)$ is infinite and a clopen subset of $X$ (since its only accumulation point is $p$). Then $C$ is the one-point compactification of the infinite space $C\cap \Bbb{I}(X)$. But $C\cap \Bbb{I}(X)$, being clopen in $X$, is $C^*$-embedded in $X$, hence the closure of $C\cap \Bbb{I}(X)$ in $\beta_0 X$ is the  space $\beta_0 (C\cap \Bbb{I}(X))$. This space could not be equal to the one-point compactiication of a discrete space, a contradiction.
\end{proof}

We also write down the following lemma giving a representation for reduced fields with a finite number of atoms as recorded in \cite[\S 26, Excercise 17]{GH}.

\begin{lemma}\label{4.8}
Let $\mathcal{A}$ be a reduced field of subsets of a set $X$. If $\text{At}(\mathcal{A})$ is finite, then
\begin{enumerate}
\item $\mathcal{A}$ is equal to the internal product $\mathcal{A}_1\oplus \mathcal{P}(F)$, where $F=\bigcup{\text{At}(\mathcal{A})}$ and 
\[\mathcal{A}_1=\{A\cap(X\setminus F): A\in \mathcal{A}\}.\]
\item $\mathcal{A}$ is complete if and only if $\mathcal{A}_1$ is complete.
\end{enumerate}
\end{lemma}


\begin{theorem}\label{4.9}
Let $\mathcal{A}$ be a $\sigma$-field of subsets of a set $X$.   $\EuScript{M}(X,\mathcal{A})/\text{Soc}\left( \EuScript{M}(X,\mathcal{A})\right)$ is self-injective if and only if $\mathcal{A}$ is a complete and $\mathfrak{c}^+$-additive Boolean algebra with a finite number of atoms.
\end{theorem}

\begin{proof}
Suppose first that $\mathcal{A}$ is an arbitrary $\sigma$-field of subsets of a set $X$. By Proposition \ref{3.13}, there is a set $Y$, a reduced $\sigma$-field $\mathcal{B}$ on $Y$, a Boolean isomorphism $\varphi: \mathcal{A}\rightarrow \mathcal{B}$ and a ring isomorphism $\theta: \EuScript{M}(X,\mathcal{A})\rightarrow \EuScript{M}(Y,\mathcal{B})$. The isomorphism $\theta$ maps the socle of $\EuScript{M}(X,\mathcal{A})$ onto the socle of $\EuScript{M}(Y,\mathcal{B})$, i.e., $\theta\left(\text{Soc}\left(\EuScript{M}(X,\mathcal{A})\right)\right)= C_F\left(Y_{\mathcal{B}}\right)$. Hence the following isomorphism holds.
\[\frac{ \EuScript{M}(X,\mathcal{A})}{\text{Soc}\left(\EuScript{M}(X,\mathcal{A})\right)}\cong \frac{ \EuScript{M}(Y,\mathcal{B})}{ C_F\left(Y_{\mathcal{B}}\right)}\]

First, we claim that the factor ring $C\left(Y_{\mathcal{B}},\Bbb{R}_{\text{disc}}\right)/ C_F(Y_{\mathcal{B}})$ is a ring of quotients  of $\EuScript{M}(Y,\mathcal{B})/ C_F\left(Y_{\mathcal{B}}\right)$. To see this, let $\bar{f}=f+ C_F\left(Y_{\mathcal{B}}\right)$ be a non-zero element of the factor ring  $C\left(Y_{\mathcal{B}},\Bbb{R}_{\text{disc}}\right)/ C_F(Y_{\mathcal{B}})$. The subset $Y\setminus Z(f)$ is infinite. We consider two cases:

Case (1): If the image of $f$ is finite, there is a nonzero real number $r\in f\left(Y\right)\setminus\{0\}$ so that $f^{\leftarrow}(r)$ is an infinite clopen subset of $Y_{\mathcal{B}}$. Since $\mathcal{B}$ is a base for $Y_{\mathcal{B}}$ and it is closed under countable union, we may choose an infinite $B\in \mathcal{B}$ such that $B\subseteq f^{\leftarrow}(r)$. The element $\chi_{_{B}}+C_F(Y_{\mathcal{B}})$ is a nonzero member of the factor ring $\EuScript{M}(Y,\mathcal{B})/ C_F\left(Y_{\mathcal{B}}\right)$ and $\left(\chi_{_{B}}+C_F(Y_{\mathcal{B}})\right)\left(f+C_F(Y_{\mathcal{B}})\right)=r\chi_{_{B}}+C_F(Y_{\mathcal{B}})$ is a nonzero element of $\EuScript{M}(Y,\mathcal{B})/ C_F\left(Y_{\mathcal{B}}\right)$, as well.

Case (2): If the image of $f$ is infinite, we may choose an infinite sequence of real numbers $\{r_n: n\in \Bbb{N}\}$ in $f(Y)\setminus\{0\}$. For each $n\in \Bbb{N}$, there is $B_n\in \mathcal{B}$ so that $B_n\subseteq f^{\leftarrow}(r_n)$. Since $\mathcal{B}$ is a $\sigma$-field, the set $B=\bigcup B_n$ is in $\mathcal{B}$ and the function $g=\sum_{_{n\in \Bbb{N}}}       \chi_{_{B_n}}$ is in $\EuScript{M}(Y,\mathcal{B})$. Since $Y\setminus Z(g)=B$, the element $g+C_F\left(Y_{\mathcal{B}}\right)$ is nonzero in the factor ring $\EuScript{M}(Y,\mathcal{B})/ C_F\left(Y_{\mathcal{B}}\right)$ and 
\[fg=\sum_{_{n\in \Bbb{N}}}r_n\chi_{_{B_n}}.\]
In view of  \cite[Lemma 6]{AHM}, $fg$ belongs to $\EuScript{M}(Y,\mathcal{B})$. Since $X\setminus Z(fg)=B$, it follows that the element 
$\left(f+C_F(Y_{\mathcal{B}})\right)\left(g+C_F(Y_{\mathcal{B}})\right)$ is a nonzero member of $\EuScript{M}(Y,\mathcal{B})/ C_F\left(Y_{\mathcal{B}}\right)$. 
We may conclude from cases (1) and (2) that $C\left(Y_{\mathcal{B}},\Bbb{R}_{\text{disc}}\right)/ C_F(Y_{\mathcal{B}})$ is a ring of quotients of $\EuScript{M}(Y,\mathcal{B})/ C_F\left(Y_{\mathcal{B}}\right)$.
\\ As  $\EuScript{M}(Y,\mathcal{B})/ C_F\left(Y_{\mathcal{B}}\right)$ is self-injective, it must be coincide with $C\left(Y_{\mathcal{B}},\Bbb{R}_{\text{disc}}\right)/ C_F(Y_{\mathcal{B}})$. Thus  $\EuScript{M}(Y,\mathcal{B})=C\left(Y_{\mathcal{B}},\Bbb{R}_{\text{disc}}\right)$ which implies that  $\mathcal{B}=\EuScript{C}\EuScript{O}(Y_{\mathcal{B}})$. In asmuch as $\mathcal{B}$ is a $\sigma$-field, the space $Y_{\mathcal{B}}$ is a $P$-space. Now, combining part (1) of Theorem \ref{4.7} with \cite[Exercise 4K.1]{GJ},  $Y_{\mathcal{B}}$ has only a finite number of isolated points. This implies that the  set of all atoms of $\mathcal{B}$ is finite. Using Lemma \ref{4.8}, if we put $F=\bigcup \text{At}(\mathcal{B})$ we obtain that 

\[\mathcal{B}=\mathcal{B}_1\oplus \mathcal{P}(F),\]
where $\mathcal{B}_1$ equals $\{B\cap\left(Y\setminus F\right): B\in \mathcal{B}\}$. Corollary \ref{4.5} implies the following isomorphism:

\[ \frac{ \EuScript{M}(Y,\mathcal{B})}{ C_F\left(Y_{\mathcal{B}}\right)}=\frac{ C\left(Y_{\mathcal{B}},\Bbb{R}_{\text{disc}}\right)}{ C_F\left(Y_{\mathcal{B}}\right)}\cong C\left(Y_{\mathcal{B}}\setminus\Bbb{I}\left(Y_{\mathcal{B}}\right),\Bbb{R}_{\text{disc}}\right)=\EuScript{M}(Y\setminus F,\mathcal{B}_1)\]

By Proposition \ref{3.2}, $\mathcal{B}_1=\EuScript{C}\EuScript{O}(Y_{\mathcal{B}}\setminus \Bbb{I}\left(Y_{\mathcal{B}}\right))$ and Theorem \ref{3.15} implies that $\mathcal{B}_1$  must be complete $\mathfrak{c}^+$-additive field of sets. Now, using part (2) of Lemma \ref{4.8}, $\mathcal{B}$  is a complete $\mathfrak{c}^+$-additive field of sets. 

For the converse, suppose that $\mathcal{B}$ is complete $\mathfrak{c}^+$-additive field of sets on $Y$ with a finite number of atoms. Then $\mathcal{B}=\EuScript{C}\EuScript{O}(Y_{\mathcal{B}})$ and $\Bbb{I}\left(Y_{\mathcal{B}}\right)$ is finite. This means that 
\[ \frac{ \EuScript{M}(Y,\mathcal{B})}{ C_F\left(Y_{\mathcal{B}}\right)}=\frac{ C\left(Y_{\mathcal{B}},\Bbb{R}_{\text{disc}}\right)}{ C_F\left(Y_{\mathcal{B}}\right)}\]
Now, applying Theorem \ref{4.7}, this factor ring is self-injective.
\end{proof}

\begin{corollary}\label{4.9}
Let $\mathcal{A}$ be an arbitrary field of subsets  of an infinite set $X$. If $X$ is of non-measurable cardinal, the factor ring  $\EuScript{M}(X,\mathcal{A})/\text{Soc}\left( \EuScript{M}(X,\mathcal{A})\right)$ is never a self-injective ring.
\end{corollary}

\begin{proof}
Regarding Proposition \ref{3.11}, we may assume that $\mathcal{A}$ is a reduced field of subsets of the set $X$. Suppose that  $\EuScript{M}(X,\mathcal{A})/\text{Soc}\left( \EuScript{M}(X,\mathcal{A})\right)$ is self-injective. Now by Theorem \ref{4.9}, $\mathcal{A}$ is a complete $\mathfrak{c}^+$-additive field of sets with only a finite number of atoms. Completeness of $\mathcal{A}$ implies that $\mathcal{A}= \EuScript{C}\EuScript{O}(X_{\mathcal{A}})$ and hence part (1) of Theorem \ref{1.2.2} implies that  $X_{\mathcal{A}}$ is extremally disconnected $P_{\mathfrak{c}^+}$- space with non-measurable cardinal and hence $X_{\mathcal{A}}$ is discrete (see \cite[Exercise 12H.6]{GJ}). Since $X_{\mathcal{A}}$ has only finite number of isolated points, it follows that $X$ is finite, a contradiction.
\end{proof}

\bibliographystyle{amsplain}

\end{document}